\documentclass[letterpaper, 10 pt, conference]{ieeeconf}

\IEEEoverridecommandlockouts                              
\usepackage[utf8]{inputenc}
\usepackage{graphicx}
\usepackage{amsmath,amsfonts}

\usepackage{cite}
\usepackage{color}
\usepackage[dvipsnames]{xcolor}
\usepackage{hyperref}
\usepackage{multirow}
\usepackage[linesnumbered,ruled]{algorithm2e}

\usepackage{soul}
\usepackage[section]{placeins}  
\usepackage{setspace}

\usepackage{etoolbox}

\makeatletter
\patchcmd{\@algocf@start}
  {-1.5em}
  {0pt}
  {}{}
\makeatother

\DeclareMathAlphabet{\mbf}{OT1}{ptm}{b}{n}
\newcommand{\mbs}[1]{{\boldsymbol{#1}}}

\newcommand{\trans}{{\ensuremath{\mathsf{T}}}} 
\newcommand{\norm}[1]{\left\Vert#1\right\Vert} 

\newcommand{\bbm}{\begin{bmatrix}}
\newcommand{\ebm}{\end{bmatrix}}
\newcommand{\bma}[1]{\left[\begin{array}{#1}}
\newcommand{\ema}{\end{array}\right]}

\newcommand{\beq}{\begin{equation}}
\newcommand{\eeq}{\end{equation}}
\newcommand{\bdis}{\begin{displaymath}}
\newcommand{\edis}{\end{displaymath}}
\newcommand{\beqarray}{\begin{eqnarray}}
\newcommand{\eeqarray}{\end{eqnarray}}
\newcommand{\beqarraynn}{\begin{eqnarray*}}
\newcommand{\eeqarraynn}{\end{eqnarray*}}

\newtheorem{theorem}{Theorem}[section]
\newtheorem{lemma}{Lemma}[section]

\newtheorem{assumption}[theorem]{Assumption}

\newtheorem{remark}[theorem]{Remark}

\newcommand{\bone}{\mbf{1}}

\DeclareMathOperator*{\argmin}{arg\,min}

\title{Robust Local Stabilization of Nonlinear Systems with Controller-Dependent Norm Bounds: A Convex Approach with Input-Output Sampling}

\author{Sze Kwan Cheah, Diganta Bhattacharjee, Maziar S. Hemati, and Ryan J. Caverly 
\thanks{This work is supported by the Office of the Under Secretary of Defense for Research and Engineering under award number FA9550-21-1-0213. DB and MSH acknowledge partial support from the Air Force Office of Scientific Research under award numbers FA9550-21-1-0106 and FA9550-22-1-0004, the Army Research Office under award number W911NF-20-1-0156, and the National Science Foundation under award number CBET-1943988. 
\textit{(SKC and DB are co-first authors.)}}
\thanks{The authors are with the Department of Aerospace Engineering \& Mechanics, University of Minnesota, Minneapolis, MN 55455, USA. Email: \{cheah013, dbhattac, mhemati, rcaverly\}@umn.edu
}
}

\begin{document}

\maketitle

\begin{abstract}
This letter presents a framework for synthesizing a robust full-state feedback controller for systems with unknown nonlinearities. Our approach characterizes input-output behavior of the nonlinearities in terms of local norm bounds using available sampled data corresponding to a known region about an equilibrium point.
A challenge in this approach is that if the nonlinearities have explicit dependence on the control inputs, an \emph{a priori} selection of the control input sampling region is required to determine the local norm bounds. This leads to a ``chicken and egg'' problem, where the local norm bounds are required for controller synthesis, but the region of control inputs needed to be characterized cannot be known prior to synthesis of the controller.
To tackle this issue, we constrain the closed-loop control inputs within the sampling region while synthesizing the controller.
As the resulting synthesis problem is non-convex, three semi-definite programs (SDPs) are obtained through convex relaxations of the main problem, and an iterative algorithm is constructed using these SDPs for control synthesis. Two numerical examples are included to demonstrate the effectiveness of the proposed algorithm.
\end{abstract}

\section{Introduction}
A large portion of the existing literature on nonlinear control comprises model-based approaches (e.g., backstepping, feedback linearization \cite{Khalil}) that depend on the availability of a sufficiently accurate analytical model of the system or plant to be controlled.
There is an inherent assumption that such an analytic model can be obtained from the fundamental laws governing the system. 
However, there are many instances where a reliable and/or accurate analytical model of the system cannot be obtained in practice. 
Robust control theory can also be used to design a stabilizing controller for a nonlinear system under the assumption that at least a nominal realization of the system, either linear time-invariant (LTI) or linear time-varying (LTV), can be obtained (see, e.g., 
\cite[Chapter 9]{buch2021finite, goh1996robust, Dullerud_Paganini_2000}).
In this setting, the system is described in 
linear fractional transformation (LFT) form, where the nominal system is connected in feedback with a `perturbation' that captures all aspects of the system that do not fit within an LTI or LTV framework (e.g., nonlinearity, uncertainty, time delays).
The control synthesis is then based on the input-output (I/O) properties of the perturbation, which are assumed to hold for some known structure and/or set of 
perturbations (e.g., integral quadratic constraints \cite{megretski1997system}, structured singular value \cite{Packard_Doyle_1993mu}). 
For a complex system, however, it is not trivial to quantify these I/O properties and ascertain if the accompanying assumptions hold, especially when the analytical form of the perturbation is not explicitly known.

Data-driven techniques are now becoming increasingly popular to overcome these issues\cite{hou2013model}. For example, strategies are being proposed for determining I/O properties from sampled data (like dissipation inequalities \cite{Romer_LCSS_2019} and passivity~
\cite{tanemura2018efficient}).
Also, data-driven approaches for robust control, with a focus on perturbations representing parametric uncertainties and unknown nonlinearity driven by states and/or parameters, are gaining popularity \cite{berberich2020robust, na2020output, Wang_et_al_2014}. 
However, the class of perturbations for which the unknown nonlinearity is a function of both the states and control inputs is relatively less studied.
Existing data-driven methods for this class of perturbations are either guaranteed to work in a small neighborhood of the equilibrium \cite{dePersis2020} or require extensive tuning of the associated control parameters to satisfy the underlying assumptions \cite{TANASKOVIC2017}. 
In light of these challenges associated with purely data-driven control, we have adopted an approach where I/O properties of the perturbation are established through data and subsequently used for control synthesis.

We consider a general nonlinear system and partition it into an LFT form, where the nominal system captures the system's LTI dynamics about an equilibrium and the perturbation comprises higher-order nonlinearities, which are functions of both the states and control inputs and are not necessarily known analytically. 
Assuming I/O samples of the perturbation are available, local norm bounds are derived. %
A robust state-feedback controller is synthesized that asymptotically stabilizes 
the nonlinear system locally within the sampled region for all perturbations satisfying the norm bounds.
Since the norm bounds, which depend on control inputs, are used in the synthesis, the synthesized controller must ensure that the control inputs remain within the sampling set. Moreover, the main synthesis problem involves non-convex constraints. We relax the main problem into three semi-definite programs (SDPs) that are solved iteratively.
To summarize, the 
contribution of this work is an iterative control synthesis method that results in local asymptotic stabilization of a nonlinear system, where the system's nonlinearities are analytically unknown, but available through sampling, and depend explicitly on 
control inputs. 

\textit{Notation:}
The symbol $\mathcal{N}_n$ is a shorthand for the set $\{ 1,2,\dots,n \}$, and $\norm{\cdot}$ denotes the 2-norm for vectors and spectral norm for matrices.
An $n$-dimensional vector of zeros with the $i$-th entry equal to one is denoted by $\bone_{n_{i}}$. We use $\mbf{A}>0$ to denote a symmetric, positive definite matrix $\mbf{A}$.
The maximum singular value of $\mbf{M}$ is denoted by $\bar{\sigma}(\mbf{M})$.  
For a given $\mbf{E} > 0$, an ellipsoid centered at the origin is denoted by $\mathcal{E}_n (\mbf{E}) = \{ \mbf{x} \in \mathbb{R}^n |  \norm{\mbf{E}^{-1} \mbf{x}} \leq 1  \} $.
Finally, $\mathcal{B}_n$ denotes the closed unit-norm ball in $\mathbb{R}^n$.


\section{Problem Formulation}

Consider a nonlinear dynamic system of the form
\begin{equation} \label{eq:original nonlinear system}
\dot{\mbf{x}} = \mbf{f}(\mbf{x}, \mbf{u}),
\end{equation}
where $\mbf{x} \in \mathbb{R}^{n_x}$ and $\mbf{u} \in \mathbb{R}^{n_u}$ are the state and control input vectors, respectively, and $\mbf{f}: \mathbb{R}^{n_x} \times \mathbb{R}^{n_u} \rightarrow \mathbb{R}^{n_x}$ is a nonlinear function subject to the following assumption.
\begin{assumption} \label{Assumptions on the original nonlinearity}
$\mbf{f}$ is not precisely known. However, there exists at least one equilibrium $(\mbf{x}_0, \mbf{u}_0)$ such that $\mbf{f}(\mbf{x}_0, \mbf{u}_0) = \mbf{0}$, and $\mbf{A} = \frac{\partial \mbf{f}}{\partial \mbf{x}} |_{(\mbf{x}_0, \mbf{u}_0)}$, $\mbf{B}_1 = \frac{\partial \mbf{f}}{\partial \mbf{u}} |_{(\mbf{x}_0, \mbf{u}_0)}$ are known.
\end{assumption}

All the developments in this letter are based on the above assumption. 
If multiple 
equilibria are known, we choose the one that is relevant for the problem at hand. 
Next, by setting $\mbf{x} = \mbf{x}_0 + \delta \mbf{x}$, $\mbf{u} = \mbf{u}_0 + \delta \mbf{u}$, we can rewrite \eqref{eq:original nonlinear system} as
\begin{equation} \label{eq:LTI system plus uncertainty}
\delta \dot{\mbf{x}} = \mbf{A} \delta \mbf{x} + \mbf{B}_1 \delta \mbf{u} + \mbs{\Delta}(\delta \mbf{x}, \delta \mbf{u}),
\end{equation}
where $\mbf{A}, \mbf{B}_1$ are as defined in Assumption \ref{Assumptions on the original nonlinearity}, and
$\mbs{\Delta} : \mathbb{R}^{n_x} \times \mathbb{R}^{n_u} \rightarrow \mathbb{R}^{n_x}$ is a function that captures the higher-order terms. Note that $\mbs{\Delta}$, 
whose analytical form is not available, is a function of both $\delta \mbf{x}, \delta \mbf{u}$. Also, we assume that $\mbs{\Delta}$ is memoryless and static.
Thus, system \eqref{eq:LTI system plus uncertainty} can be expressed in an LFT form having the nominal LTI dynamics $\delta \dot{\mbf{x}} = \mbf{A} \delta \mbf{x} + \mbf{B}_1 \delta \mbf{u} $ and the perturbation $\mbs{\Delta}$.

Although $\mbs{\Delta}$ is not known analytically, we assume that a finite number of input-output samples of $\mbs{\Delta}$ are available from either experiments or high-fidelity numerical simulations. 
We also assume that the sampling is carried out in a known region around the equilibrium, such as over a $N$-point grid where $N$ is large.  
We refer to this region as the sampling region $\mathcal{S} = \mathbb{X} \times \mathbb{U},$ where $\mathbb{X} \subset \mathbb{R}^{n_x}$ and $\mathbb{U} \subset \mathbb{R}^{n_u}$ are known, compact sets that contain the respective origins in the interiors.
Therefore, we have access to $N$ input-output samples $\{ \left(\delta \mbf{x}^{(k)}, \delta \mbf{u}^{(k)} \right),  \mbs{\Delta}\left(\delta \mbf{x}^{(k)}, \delta \mbf{u}^{(k)} \right) \}_{k \in \mathcal{N}_N}$,
where $\left(\delta \mbf{x}^{(k)}, \delta \mbf{u}^{(k)} \right) \in \mathcal{S}$ for all $k \in \mathcal{N}_N$.
Upon investigating these samples, we can deduce the following: 
\begin{itemize}
\item If the vector $\mbs{\Delta}\left(\delta \mbf{x}^{(k)}, \delta \mbf{u}^{(k)} \right)$ corresponding to $\delta \mbf{x}^{(k)} \neq \mbf{0}$ and $\delta \mbf{u}^{(k)} \neq \mbf{0}$ contains elements that are identically 0, then the original system \eqref{eq:original nonlinear system} has states that are governed by purely LTI dynamics. This, in conjunction with the nonzero entries (say, $n_w$ of those), can be used to find a realization of $\mbs{\Delta}\left(\delta \mbf{x}^{(k)}, \delta \mbf{u}^{(k)} \right)$ of the form
$\mbs{\Delta}\left(\delta \mbf{x}^{(k)}, \delta \mbf{u}^{(k)} \right) 
= \bbm
\Delta_1 & \cdots & 0 & \Delta_{n_{w}} & \cdots
\ebm^\trans,$
where $\Delta_i: \mathbb{R}^{n_x} \times \mathbb{R}^{n_u} \rightarrow \mathbb{R}, \ i \in \mathcal{N}_{n_w}$ are the nonlinearties corresponding to the nonzero entries. Thus, the above can be reformulated as $\mbs{\Delta}\left(\delta \mbf{x}^{(k)}, \delta \mbf{u}^{(k)} \right) = \mbf{B}_2 \mbf{w}^{(k)}$ where $\mbf{w}^{(k)} = \bbm w_1^{(k)} & w_2^{(k)} & \ldots & w_{n_{w}}^{(k)} \ebm^\trans \in \mathbb{R}^{n_w}$ contains outputs of all the nonlinearities and $\mbf{B}_2 \in \mathbb{R}^{n_x \times n_w}$ properly distributes the elements of $\mbf{w}^{(k)}$.
\item  We can also identify the individual states and control inputs that drive each nonlinear function $\Delta_i$. Based on this, the sampled input to $\Delta_i$ takes the form 
$\mbf{v}_i^{(k)} = \left[
\delta x_1^{(k)} \ \cdots \ 0 \ \cdots \ \delta u_1^{(k)}  \ \cdots \ 0 \ \cdots \right]^\trans \in \mathbb{R}^{n_x + n_u}$
which is equivalent to $\mbf{v}_i^{(k)} = \mbf{C}_{i}  \delta \mbf{x}^{(k)} + \mbf{D}_{i} \delta \mbf{u}^{(k)}$ with $\mbf{C}_{i} \in \mathbb{R}^{(n_x+n_u) \times n_x}, \ \mbf{D}_{i} \in \mathbb{R}^{(n_x+n_u) \times n_u}$ known. Therefore, for each $k \in \mathcal{N}_N$, we have
$w_i^{(k)} = \Delta_i (\mbf{v}_i^{(k)}), \ i \in \mathcal{N}_{n_w}$. 
Although expressed in this form, it is understood that each 
$\Delta_i$ maps $\mathbb{R}^{n_x} \times \mathbb{R}^{n_u}$ to $\mathbb{R}$.
\item The samples can be used to prescribe \textit{empirical norm bounds} on the inputs-outputs of each 
$\Delta_i$. Specifically, we intend to find $\gamma_i > 0$ for each $i \in \mathcal{N}_{n_w}$ such that $\left( w_i^{(k)} \right)^2 \leq \gamma_i^2 \norm{\mbf{v}_i^{(k)}}^2$
holds for all $k \in \mathcal{N}_N$.
To this end, for each pair $(i,k) \in \mathcal{N}_{n_w} \times \mathcal{N}_N$,
we define $\gamma_i^{(k)} = \left(\left( w_i^{(k)} \right)^2  \big/ \norm{\mbf{v}_i^{(k)}}^2 \right)^{1/2}$ and stack all such bounds in a matrix 
$\Gamma = \begin{bmatrix}
\Gamma_1 & \Gamma_2 & \cdots & \Gamma_{n_{w}}
\end{bmatrix},$
where $\Gamma_i = \begin{bmatrix}
\gamma^{(1)}_i & \gamma^{(2)}_i & \dots & \gamma^{(N)}_i
\end{bmatrix}^\trans$.
The empirical bounds are then specified as the maximum over each column, 
i.e., $\gamma_i = \max \Gamma_{i}$. If we seek the bounds over a region $\mathbb{D} \subseteq \mathcal{S}$, the same procedure can be repeated with the sampled data corresponding to $\mathbb{D}$.
\end{itemize}
Using the above information, system \eqref{eq:LTI system plus uncertainty} can be rewritten as 
\begin{align}
    \delta \dot{\mbf{x}} 
    &= \mbf{A} \delta \mbf{x} + \mbf{B}_1 \delta \mbf{u} + \mbf{B}_2 \mbf{w}, \nonumber
    \\
    w_i &= \Delta_i (\mbf{v}_i), \ i \in \mathcal{N}_{n_w}, \label{eq:open-loop system}
    \\
    \mbf{v}_i &= \left( \mbf{C}_{i}  \delta \mbf{x} + \mbf{D}_{i} \delta \mbf{u} \right), \ i \in \mathcal{N}_{n_w}, \nonumber
\end{align}
with the following standing assumption for a given $\mathbb{D} \subseteq \mathcal{S}$. 
\begin{assumption} \label{Assumption: Delta_i satisfies empirical bounds}
Let $\gamma_i, \ i \in \mathcal{N}_{n_w}$ be the empirical bounds corresponding to $\mathbb{D}$. Each input-output tuple $\left( \mbf{v}_i, w_i \right)$ satisfies the bound $\gamma_i$ in $\mathbb{D}$, i.e., $w_i^2 \leq \gamma_i^2 \norm{\mbf{v}_i}^2$, for all $\left( \delta \mbf{x}, \delta \mbf{u} \right) \in \mathbb{D}$ and for all $i \in \mathcal{N}_{n_{w}}$. Also, for each $i \in \mathcal{N}_{n_{w}}$, let $\mbs{\Delta}_{\mathbb{D}_{i}}$ be the set of functions $\Delta_i: \mathbb{D} \rightarrow \mathbb{R}$ for which the bound $\gamma_i$ holds.
\end{assumption}

The above assumption is reasonable since our knowledge is restricted to the extent provided by the sampled data and analytical forms of $\Delta_i$s are unknown. 
%
%
Now, we are interested in designing a state-feedback control law $\delta \mbf{u} = \mbf{K} \delta \mbf{x}$ for the open-loop system \eqref{eq:open-loop system}, under Assumption \ref{Assumption: Delta_i satisfies empirical bounds} with $\mathbb{D} = \mathbb{X}_c \times \mathbb{U}_c$ for some $\mathbb{X}_c \subseteq \mathbb{X}$, $\mathbb{U}_c \subseteq \mathbb{U}$ containing the respective origins in the interiors.
The closed-loop system thus becomes 
\begin{align}
    \delta \dot{\mbf{x}} 
        &= \left( \mbf{A}  + \mbf{B}_1 \mbf{K} \right) \delta \mbf{x} + \mbf{B}_2 \mbf{w}, \nonumber
    \\ 
    w_i &= \Delta_i (\mbf{v}_i) , \ i \in \mathcal{N}_{n_w},
    \label{eq:closed-loop system}
    \\ 
    \mbf{v}_i &= \left( \mbf{C}_{i} + \mbf{D}_{i} \mbf{K} \right) \delta \mbf{x}, \ i \in \mathcal{N}_{n_w} \nonumber.
\end{align}
The goal is to synthesize 
$\mbf{K}$ to certify the closed-loop system \eqref{eq:closed-loop system} asymptotically stable in the largest local region 
$\mathbb{X}_c \subseteq \mathbb{X}$, for a choice of $\mathbb{U}_c \subseteq \mathbb{U}$ and the corresponding bounds $\gamma_i$ and set of functions $\mbs{\Delta}_{\mathbb{D}_{i}}$, $i \in \mathcal{N}_{n_w}$.
However, as the synthesis uses $\gamma_i$, we need to verify whether the closed-loop control trajectories satisfy $\delta \mbf{u} = \mbf{K} \delta \mbf{x} \in \mathbb{U}_c$ for $\mbf{K}$ to be 
consistent with the data. 
This requires the use of an iterative approach, which is described in detail in the synthesis presented in the next section.
Before discussing the control synthesis, we outline a few useful matrix inequality results.
\begin{lemma}
\label{lemma:young}
(Young's Relation \cite{CaverlyLMI}): Consider $\mbf{X} \in \mathbb{R}^{m \times n}$ and $\mbf{Y} \in \mathbb{R}^{m \times n}$. For any $\mbf{S} >0$, it holds that 
\begin{align}
    \label{eq:young}
    \mbf{X}^\trans \mbf{Y} + \mbf{Y}^\trans \mbf{X}
    \leq 
    \mbf{X}^\trans \mbf{S}^{-1} \mbf{X} +
    \mbf{Y}^\trans \mbf{S} \mbf{Y}.
\end{align}
\end{lemma}
A special case of Lemma \ref{lemma:young} that will prove useful in the control synthesis presented in the next section considers $\mbf{S}=\mbf{I}$, $\mbf{X}=\mbf{H} \in \mathbb{R}^{n_x \times n_x}$ and $\mbf{Y}=\mbf{H}_0 \in \mathbb{R}^{n_x \times n_x}$, which leads to
\begin{align}
    \mbf{H}^\trans \mbf{H} &\geq  \mbf{H}^\trans \mbf{H}_0 + \mbf{H}^\trans_0 \mbf{H} - \mbf{H}^\trans_0 \mbf{H}_0. \label{eq:SpecialYoung}
\end{align}


\section{Control Synthesis}  
This section describes the synthesis of a static state-feedback controller $\mbf{K}$, which requires breaking the main synthesis problem into different sub-problems and iterating over these sub-problems to obtain a controller that is certified to render the closed-loop system \eqref{eq:closed-loop system} asymptotically stable within $\mathbb{X}_c$.
We will start by specifying structures of the sets $\mathbb{X}_c$ and $\mathbb{U}_c$ that will be used in the remainder of the letter. 
The local region $\mathbb{X}_c$  is taken to be a family of ellipsoids parameterized by $\mbf{W} >0$, i.e., $\mathbb{X}_c = \mathcal{E}_{n_x}(\mbf{W})$, where $\mbf{W}$ is chosen appropriately such that $\mathbb{X}_c \subseteq \mathbb{X}$. 
Similarly, we take $\mathbb{U}_c = r \mathcal{B}_{n_{u}}$, where $r>0$ is the parameter related to the control input magnitude and is chosen such that $\mathbb{U}_c \subseteq \mathbb{U}$. With these sets defined, the main synthesis problem is summarized in the next result.
\begin{theorem}
\label{theorem:revGrand}
Let $\mbf{W} >0$ and $r>0$ be chosen such that $\mathbb{X}_c = \mathcal{E}_{n_{x}}(\mbf{W}) \subseteq \mathbb{X}$ and $\mathbb{U}_c = r \mathcal{B}_{n_{u}} \subseteq \mathbb{U}$, respectively.
Suppose Assumption \ref{Assumption: Delta_i satisfies empirical bounds} holds with $\mathbb{D} = \mathbb{X}_c \times \mathbb{U}_c$. 
Then, the closed-loop system \eqref{eq:closed-loop system} is locally asymptotically stable in $\mathbb{X}_c$ for all $\Delta_i \in \mbs{\Delta}_{\mathbb{D}_{i}}, \ i \in \mathcal{N}_{n_w}$, if there exist $\mbf{P} > 0$, $\mbf{K} \in \mathbb{R}^{n_u \times n_x}$, $\tau >0$, and $\lambda_i >0, \ i \in \mathcal{N}_{n_w}$,
such that
\begin{align}
    \bbm 
        \mbf{P}(\mbf{A}+\mbf{B}_1 \mbf{K})+(\mbf{A}+\mbf{B}_1 \mbf{K})^\trans \mbf{P} & \mbf{PB}_2 & \mbs{\Theta} \\
        \mbf{B}_2^\trans \mbf{P} & \mbs{\Lambda} & \mbf{0} \\
        \mbs{\Theta}^\trans & \mbf{0} & \mbs{\Xi}
    \ebm
    &<0, \label{eq:revTheorem1}
    \\ 
    \bbm
        \tau^2 \mbf{I} & \mbf{K} \\ 
        \mbf{K}^\trans & \mbf{W}^{-1} \mbf{W}^{-1}
    \ebm &\geq 0, \label{eq:revTheorem2}
    \\
    \tau &\leq r, \label{eq:revTheorem3}
\end{align}
where $\mbs{\Lambda} = -\textnormal{diag} \left( \lambda_1, \ldots, \lambda_{n_{w}} \right)$, and 
\small{
\begin{equation} \label{eq:Xi and Theta}
\mbs{\Xi} = -\textnormal{diag}(\frac{\lambda_1}{\gamma_1^2} \mbf{I}, \dots,\frac{\lambda_{n_{w}}}{\gamma_{n_{w}}^2} \mbf{I}), \
\mbs{\Theta} = \bbm 
\lambda_1 \mbs{\Phi}_1,\dots, \lambda_n \mbs{\Phi}_{n_{w}} \ebm,
\end{equation}
}
\normalsize
with $\mbs{\Phi}_i = \mbf{C}_{i}^\trans + \mbf{K}^\trans \mbf{D}_{i}^\trans$, $i \in \mathcal{N}_{n_w}$.
\end{theorem}
\begin{proof}
We establish the proof in three parts: first, we derive a condition that ensures asymptotic stability of the closed-loop system within the local region $\mathbb{X}_c$ under the assumption that $\delta \mbf{u}= \mbf{K} \delta \mbf{x} \in \mathbb{U}_c$ for all $\delta \mbf{x} \in \mathbb{X}_c$; then, we obtain the equivalent stability condition \eqref{eq:revTheorem1}; finally, we constrain the control signal such that $\delta \mbf{u}= \mbf{K} \delta \mbf{x} \in \mathbb{U}_c$ for all $\delta \mbf{x} \in \mathbb{X}_c$, which leads to \eqref{eq:revTheorem2} and \eqref{eq:revTheorem3}.
 
\textit{Part-1:}
Define the candidate Lyapunov function $V = \delta \mbf{x}^\trans \mbf{P} \delta \mbf{x}$ with $\mbf{P}  > 0$. Taking the time-derivative of $V$ and using \eqref{eq:closed-loop system} results in 
\begin{align*}
    \dot{V} &= 
    \delta \mbf{x}^\trans \mbf{P} \delta \dot{\mbf{x}}
    + \delta \dot{\mbf{x}}^\trans \mbf{P}\delta \mbf{x} \\
    &=
    \bbm
      \delta \mbf{x} \\ \mbf{w}
    \ebm^\trans 
    \bbm
        \mbf{P}(\mbf{A}+\mbf{B}_1 \mbf{K})+(\mbf{A}+\mbf{B}_1 \mbf{K})^\trans \mbf{P} & \mbf{PB}_{2}\\
        \mbf{B}_2^\trans \mbf{P} & \mbf{0} 
    \ebm
    \bbm \delta \mbf{x} \\ \mbf{w} \ebm. 
\end{align*}
The inputs and outputs of each $\Delta_i$ can be rewritten using \eqref{eq:closed-loop system} as
\begin{align}
    \bbm \mbf{v}_i \\ w_i \ebm 
    &= 
    \bbm
        (\mbf{C}_{i} + \mbf{D}_{i} \mbf{K}) & \mbf{0} \\
        \mbf{0} & \mbf{1}_{n_{w_{i}}}^\trans
    \ebm
    \bbm
        \delta \mbf{x} \\ \mbf{w}
    \ebm. 
    \label{eq:translation}
\end{align}
Now, suppose $\delta \mbf{u} = \mbf{K} \delta \mbf{x} \in \mathbb{U}_c$ holds for all $\delta \mbf{x} \in \mathbb{X}_c$. Then, under Assumption \ref{Assumption: Delta_i satisfies empirical bounds} with $\mathbb{D}= \mathbb{X}_c \times \mathbb{U}_c$, we have $w_i^2 \leq \gamma_i^2 \norm{\mbf{v}_i}^2$ for each $\Delta_i \in \mbs{\Delta}_{\mathbb{D}_{i}}$, $i \in \mathcal{N}_{n_w}$, and for all $\delta \mbf{x} \in \mathbb{X}_c$, which is equivalently given by
\begin{equation}
    \bbm \mbf{v}_i \\ {w}_i \ebm^\trans
    \bbm \gamma_i^2 \mbf{I} & \mbf{0} \\ \mbf{0} & -1 \ebm
    \bbm \mbf{v}_i \\ {w}_i \ebm
    \geq 0. 
    \label{eq:inputoutputQC}
\end{equation}
This, along with \eqref{eq:translation}, leads to a quadratic constraint (QC)
in $\delta \mbf{x}$ and $\mbf{w}$ for each $i \in \mathcal{N}_{n_w}$, expressed as
\begin{align}
    \bbm
      \delta \mbf{x} \\ \mbf{w}
    \ebm^\trans
    \bbm 
        \gamma_i^2 \mbs{\Phi}_i  \mbs{\Phi}_i^\trans & \mbf{0} \\
        \mbf{0} & -\mbf{1}_{n_{w_{i}}} \mbf{1}_{n_{w_{i}}}^\trans
    \ebm
    \bbm
      \delta \mbf{x} \\ \mbf{w}
    \ebm
    &\geq 0,
    \label{eq:normbound}
\end{align}
where $\mbs{\Phi}_i = \mbf{C}_{i}^\trans + \mbf{K}^\trans \mbf{D}_{i}^\trans$. In the current setting, all QCs of the form \eqref{eq:normbound} for $\Delta_i \in \mbs{\Delta}_{\mathbb{D}_{i}}$, $i \in \mathcal{N}_{n_w}$ hold for all $\delta \mbf{x} \in \mathbb{X}_c$.
This would imply, through the S-procedure, 
that $\dot{V} < 0$ for all $\delta \mbf{x} \in \mathbb{X}_c$ if there exists $\lambda_i \geq 0, \ i \in \mathcal{N}_{n_w}$ such that
\begin{equation}
\begin{split}
\bbm 
    \mbf{P}(\mbf{A}+\mbf{B}_1 \mbf{K})+(\mbf{A}+\mbf{B}_1 \mbf{K})^\trans \mbf{P} & \mbf{PB}_2 \\
    \mbf{B}_2^\trans \mbf{P} & \mbf{0}
    \ebm \\
    +
    \sum_{i=1}^{n_w} 
    \lambda_i 
    \bbm
        \gamma_i^2 \mbs{\Phi}_i  \mbs{\Phi}_i^\trans & \mbf{0} \\
        \mbf{0} & -\mbf{1}_{n_{w_{i}}} \mbf{1}_{n_{w_{i}}}^\trans 
    \ebm
    < 0 .\label{eq:stability_condition_1}
\end{split}
\end{equation}
This concludes the first part of the proof where we have derived a condition for local asymptotic stability.

\textit{Part-2:}
Making the restriction that $\lambda_i > 0$ for all $i \in \mathcal{N}_{n_w}$, allows for \eqref{eq:stability_condition_1} to be rewritten as
\begin{equation}  \label{eq:stability_norm_combined}
\begin{split}
\bbm
        \mbf{P}(\mbf{A}+\mbf{B}_1 \mbf{K})+(\mbf{A}+\mbf{B}_1 \mbf{K})^\trans \mbf{P} & \mbf{PB}_{2}\\
        \mbf{B}_2^\trans \mbf{P} & \mbs{\Lambda} 
    \ebm \\
    -
    \sum_{i=1}^{n_w} 
    \bbm 
        -\lambda_i \mbs{\Phi}_i 
        \frac{\gamma^2_i}{\lambda_i} 
        \mbs{\Phi}_i^\trans \lambda_i & \mbf{0} \\
        \mbf{0} & \mbf{0}
    \ebm <0,
\end{split}
\end{equation}
where $\mbs{\Lambda} = -\sum_{i=1}^{n_w} \lambda_i \mbf{1}_{n_{w_{i}}} \mbf{1}_{n_{w_{i}}}^\trans = - \textnormal{diag} \left( \lambda_1, \ldots, \lambda_{n_{w}} \right)$. Each matrix in the sum above can be expressed as
\begin{equation*}
\bbm 
        -\lambda_i \mbs{\Phi}_i 
        \frac{\gamma^2_i}{\lambda_i} 
        \mbs{\Phi}_i^\trans \lambda_i & \mbf{0} \\
        \mbf{0} & \mbf{0}
    \ebm = \bbm \lambda_i \mbs{\Phi}_i \\ \mbf{0} \ebm \left(-\frac{\gamma^2_i}{\lambda_i} \mbf{I}  \right) \bbm \lambda_i \mbs{\Phi}^\trans_i & \mbf{0} \ebm,
\end{equation*}
which leads to
\begin{equation}
\begin{split}
\sum_{i=1}^{n_w} 
    \bbm 
        -\lambda_i \mbs{\Phi}_i 
        \frac{\gamma^2_i}{\lambda_i} 
        \mbs{\Phi}_i^\trans \lambda_i & \mbf{0} \\
        \mbf{0} & \mbf{0}
    \ebm 
    = \bbm \mbs{\Theta} \\ \mbf{0} \ebm (\mbs{\Xi})^{-1} \bbm \mbs{\Theta} \\ \mbf{0} \ebm^\trans,
\end{split}
\end{equation} 
where $\mbs{\Xi}$ and $\mbs{\Theta}$ are as shown in \eqref{eq:Xi and Theta}.
Thus, \eqref{eq:stability_norm_combined} becomes
\begin{equation} \label{eq:stability_norm_combined_1}
\begin{split}
\bbm
        \mbf{P}(\mbf{A}+\mbf{B}_1 \mbf{K})+(\mbf{A}+\mbf{B}_1 \mbf{K})^\trans \mbf{P} & \mbf{PB}_{2}\\
        \mbf{B}_2^\trans \mbf{P} & \mbs{\Lambda} 
    \ebm \\
    - \bbm \mbs{\Theta} \\ \mbf{0} \ebm (\mbs{\Xi})^{-1} \bbm \mbs{\Theta} \\ \mbf{0} \ebm^\trans < 0.
\end{split}
\end{equation}
Applying the Schur complement to \eqref{eq:stability_norm_combined_1} leads to 
\eqref{eq:revTheorem1}, which completes the proof for the equivalent stability condition.

\textit{Part-3:}
To obtain a controller $\mbf{K}$ that satisfies the assumption used in \textit{Part-1} (i.e., $\delta \mbf{u} = \mbf{K} \delta \mbf{x} \in \mathbb{U}_c$ holds for all $\delta \mbf{x} \in \mathbb{X}_c$) and is, therefore, consistent with the values of $\gamma_i$ used, we need to ensure $\norm{\delta\mbf{u}} = \norm{\mbf{K} \delta \mbf{x}} \leq r$. 
A bound on $\norm{\delta \mbf{u}}$ is found using the definition of $\bar{\sigma}(\cdot)$ 
and knowing that $\norm{\mbf{W}^{-1} \delta \mbf{x}} \leq 1$, for all $\delta \mbf{x} \in \mathbb{X}_c$, which yields
\begin{equation}
\begin{split}
    \norm{\delta \mbf{u}} = \norm{\mbf{K} \delta \mbf{x}}
    = \norm{\mbf{KWW}^{-1} \delta \mbf{x}} \\
    \leq \bar{\sigma}(\mbf{KW}) \norm{\mbf{W}^{-1} \delta \mbf{x}} \leq \bar{\sigma}(\mbf{KW}). 
\end{split}
\end{equation} 
Thus, ensuring $\bar{\sigma}(\mbf{KW}) \leq r$ guarantees that $\norm{\delta \mbf{u}} \leq r$. We first find a $\tau > 0$ such that $\bar{\sigma}(\mbf{KW}) \leq \tau$. This condition can be expressed equivalently as $\tau^2 \mbf{I} ~\geq (\mbf{K} \mbf{W} \mbf{W} \mbf{K}^\text{T})$, applying the Schur complement to which leads to \eqref{eq:revTheorem2}. Finally, specifying $\tau \leq r$ means $\norm{\delta \mbf{u}} \leq \bar{\sigma}(\mbf{KW}) \leq \tau \leq r$. This completes the last part of the proof.
\end{proof}
\begin{remark}
Note that although we have considered the 2-norm bound in this letter, Theorem \ref{theorem:revGrand} is suitable for other quadratic characterizations of I/O behavior (of the form \eqref{eq:inputoutputQC}) of the nonlinearities (e.g., weighted 2-norm bounds). This would involve suitably modifying the stability condition \eqref{eq:revTheorem1}.
\end{remark}
\begin{remark}
Theorem \ref{theorem:revGrand} holds for any $\mathbb{X}_c = \mathcal{E}_{n_{x}}(\mbf{W}) \subseteq \mathbb{X}$ and $\mathbb{U}_c = r \mathcal{B}_{n_{u}} \subseteq \mathbb{U}$, with the values of $\gamma_i$ computed from the sampled data corresponding to the chosen region $\mathbb{D} = \mathbb{X}_c \times \mathbb{U}_c$. 
Hence, one could ideally select $\mbf{W}>0$ such that $\mathbb{X}_c$ is the largest ellipsoid contained in $\mathbb{X}$, choose a $r>0$, and find a feasible point satisfying all the constraints in Theorem~\ref{theorem:revGrand} to obtain a
controller which would asymptotically stabilize the closed-loop system \eqref{eq:closed-loop system} for all initial conditions within this ellipsoid.
However, it is not trivial to find such a feasible point for the constraints involved.
Specifically, the matrix inequality \eqref{eq:revTheorem1} is non-convex,
as it is bilinear in the variables $\mbf{P},\ \mbf{K}$. 
Also, $\mbs{\Theta}$ is quadratic in the variables $\lambda_i, \ \mbf{K}$. The bilinearity issue is well-known and can be addressed by applying a congruence transformation and a change of variables (see, e.g., \cite[p.~119]{CaverlyLMI}). However, the matrix inequality in~\eqref{eq:revTheorem2} then becomes non-convex in the transformed variables. We address these issues by reformulating and/or relaxing these constraints into convex ones. The convex constraints are then used to set up three different semi-definite programs (SDPs) for the controller synthesis. 
\end{remark}
We start by deriving an alternative form of \eqref{eq:revTheorem1} using a congruence transformation 
with $\text{diag}\left( \mbf{P}^{-1}, \mbf{I}, \mbf{I} \right)$ as
\begin{align}
        \bbm 
            \mbf{A} \mbf{R} + \mbf{B}_1 \mbf{F} + \mbf{R} \mbf{A}^\trans + \mbf{F}^\trans \mbf{B}_1^\trans & \mbf{B}_2
            & \tilde{\mbs{\Theta}} \\
            \mbf{B}_2^\trans & \mbs{\Lambda} & \mbf{0} \\
            \tilde{\mbs{\Theta}}^\trans & \mbf{0} & \mbs{\Xi}
        \ebm <0,
    \label{eq:step3}
\end{align}
where $\mbf{R} = \mbf{P}^{-1}$, $\mbf{F} = \mbf{K} \mbf{P}^{-1} = \mbf{K R}$,
and
$\tilde{\mbs{\Theta}} = 
\bbm 
    \lambda_1 \tilde{\mbs{\Phi}}_1,\dots, \lambda_{n_{w}} \tilde{\mbs{\Phi}}_{n_{w}} 
\ebm$ 
with $\tilde{\mbs{\Phi}}_i = \mbf{R} \mbf{C}_{i}^\trans + \mbf{F}^\trans \mbf{D}_{i}^\trans$. Note that \eqref{eq:step3} is not an LMI in the variables $\mbf{R}, \mbf{F}, \lambda_i$. However, if the values of $\lambda_i$s are known or given, \eqref{eq:step3} is an LMI in $\mbf{R}, \mbf{F}$. 
Another approach for deriving an LMI form of \eqref{eq:step3} is by setting $\lambda_i = \lambda$ and 
applying a congruence transformation on \eqref{eq:step3} with
$\text{diag} ( \sqrt{\lambda}\mbf{I}, (1/\sqrt{\lambda}) \mbf{I}, (1/\sqrt{\lambda}) \mbf{I} )$. These steps lead to
\begin{align}
\bbm 
    \mbf{A} \mbf{R} + \mbf{B}_1 \mbf{F} + 
    \mbf{R} \mbf{A}^\trans + \mbf{F}^\trans \mbf{B}_1^\trans 
    & \mbf{B}_2 & \bar{\mbs{\Theta}}
    \\ \mbf{B}_2^\trans & -\mbf{I} & \mbf{0} 
    \\ \bar{\mbs{\Theta}}^\trans & \mbf{0} & \bar{\mbs{\Xi}}
\ebm < 0,
\label{eq:step1}
\end{align}
where the $\lambda$ is absorbed into the definitions of $\mbf{R}$ and $\mbf{F}$ 
(i.e., $\mbf{R} = \lambda \mbf{R}$, $\mbf{F} = \mbf{K} \lambda \mbf{R}$) for consistent notation, 
$\bar{\mbs{\Xi}} = \text{diag}(-\frac{1}{\gamma_1^2} \mbf{I},\dots,-\frac{1}{\gamma_{n_{w}}^2} \mbf{I})$, and
$\bar{\mbs{\Theta}} = 
\bbm 
     \tilde{\mbs{\Phi}}_1, \dots, \tilde{\mbs{\Phi}}_{n_{w}} 
\ebm $.

Because of the change of variables introduced above, we need to suitably modify the constraint $\bar{\sigma}(\mbf{KW}) \leq \tau$. To this end, we state our first convex reformulation of \eqref{eq:revTheorem2} next. 
\begin{lemma} \label{lemma:step1_Ksize}
Let $\mbf{W} > 0$ be given and  $\mbf{K} = \mbf{F} \mbf{R}^{-1}$. Then, 
$\bar{\sigma}(\mbf{KW}) \leq \sqrt{\beta}$ if there exists $\beta > 0$ such that
    \begin{align}
    \bbm
        \beta \mbf{I} & \mbf{F} \\
        \mbf{F}^\trans & ( \mbf{W}^{-1} \mbf{R})^\trans + ( \mbf{W}^{-1} \mbf{R}) - \mbf{I}
    \ebm
    \geq 0.
    \label{eq:step1_Ksize}
\end{align}
\end{lemma}
\begin{proof}
Performing a congruence transformation with $\text{diag}(\mbf{I}, \mbf{R})$ on \eqref{eq:revTheorem2} results in 
\begin{equation}
    \bbm
        \tau^2 \mbf{I} & \mbf{F} \\
        \mbf{F}^\trans &  ( \mbf{W}^{-1} \mbf{R})^\trans ( \mbf{W}^{-1} \mbf{R})
    \ebm \geq 0 . \label{eq:sigma_bar_ineq}
\end{equation}
Using Lemma \ref{lemma:young} with $\mbf{S}=\mbf{X}=\mbf{I}$ and $\mbf{Y}=\mbf{W}^{-1} \mbf{R}$, we relax the bilinear term $( \mbf{W}^{-1} \mbf{R})^\trans ( \mbf{W}^{-1} \mbf{R})$ as
\begin{align}
    ( \mbf{W}^{-1} \mbf{R})^\trans ( \mbf{W}^{-1} \mbf{R}) 
    &\geq ( \mbf{W}^{-1} \mbf{R})^\trans + ( \mbf{W}^{-1} \mbf{R}) - \mbf{I}.
\end{align}
Therefore, \eqref{eq:sigma_bar_ineq} is implied by \eqref{eq:step1_Ksize}, which ensures that $\bar{\sigma}(\mbf{KW}) \leq \tau$. Defining $\beta = \tau^2$ completes the proof. 
\end{proof}
Alternatively, we can find a different relaxation of the bilinear term $( \mbf{W}^{-1} \mbf{R})^\trans ( \mbf{W}^{-1} \mbf{R})$ for a given $\mbf{R}_0$. 
This is analogous to linearizing the bilinear term about $\mbf{R}_0$, which is similar to the convex overbounding approach in \cite{warner2017iterative}.
\begin{lemma}
\label{lemma:smallK_step3}
Let $\mbf{W} > 0$, $\mbf{R}_0 > 0$ be given and $\mbf{K} = \mbf{F} \mbf{R}^{-1}$. Then, 
$\bar{\sigma}(\mbf{KW}) \leq \sqrt{\beta}$ if there exists $\beta > 0$ such that
\begin{align}
    \bbm
        \beta \mbf{I} & \mbf{F} \\ 
        \mbf{F}^\trans & \mbf{T}_1
    \ebm 
    \geq 0,
    \label{eq:smallK_step3}
\end{align}
where 
$\mbf{T}_1 = (\mbf{W}^{-1} \mbf{R})^\trans (\mbf{W}^{-1} \mbf{R}_0) + (\mbf{W}^{-1} \mbf{R}_0)^\trans (\mbf{W}^{-1} \mbf{R}) - (\mbf{W}^{-1} \mbf{R}_0)^\trans (\mbf{W}^{-1} \mbf{R}_0)$.
\end{lemma}
\begin{proof}
We have already established that 
$\bar{\sigma}(\mbf{KW}) \leq \tau \iff \eqref{eq:sigma_bar_ineq}$.
Then, using \eqref{eq:SpecialYoung} with $\mbf{H}=\mbf{W}^{-1}\mbf{R}$ and $\mbf{H}_0=\mbf{W}^{-1}\mbf{R}_0$ leads to 
\begin{multline}
( \mbf{W}^{-1} \mbf{R})^\trans ( \mbf{W}^{-1} \mbf{R}) 
\geq (\mbf{W}^{-1} \mbf{R})^\trans (\mbf{W}^{-1} \mbf{R}_0) \\ + (\mbf{W}^{-1} \mbf{R}_0)^\trans (\mbf{W}^{-1} \mbf{R}) 
- (\mbf{W}^{-1} \mbf{R}_0)^\trans (\mbf{W}^{-1} \mbf{R}_0).
\end{multline}
Therefore, \eqref{eq:sigma_bar_ineq} is implied by \eqref{eq:smallK_step3}, which ensures that $\bar{\sigma}(\mbf{KW}) \leq \tau$. Denoting $\beta = \tau^2$ completes the proof.
\end{proof}
Finally, to express \eqref{eq:smallK_step3} in terms of the variables $\mbf{K}$ and $\mbf{P}$, instead of $\mbf{F}$ and $\mbf{R}$, we perform a congruence transformation with $\text{diag} \left( \mbf{I}, \mbf{P} \right)$ on \eqref{eq:smallK_step3} to obtain 
\begin{equation} \label{eq:intermediate_smallK_step2}
    \bbm
        \beta \mbf{I} & \mbf{K} \\ 
        \mbf{K}^\trans & \mbf{T}_2  
    \ebm 
    \geq 0 ,
\end{equation}
where $\mbf{T}_2 = \mbf{W}^{-1} \mbf{W}^{-1} \mbf{R}_0 \mbf{P} + \mbf{P} (\mbf{W}^{-1} \mbf{R}_0)^\trans \mbf{W}^{-1} - 
\mbf{P} (\mbf{W}^{-1} \mbf{R}_0)^\trans (\mbf{W}^{-1} \mbf{R}_0) \mbf{P}$. The Schur complement is then applied to \eqref{eq:intermediate_smallK_step2} to yield
\begin{align}
\bbm
        \beta \mbf{I} & \mbf{K} & \mbf{0}\\ 
        \mbf{K}^\trans & \mbf{T}_3 & \mbf{P} \\
        \mbf{0} & \mbf{P} & 
        \mbf{R}_0^{-1} \mbf{W} \mbf{W} \mbf{R}_0^{-1}
    \ebm 
    & \geq 0,
    \label{eq:smallK_step2}
\end{align}
where $\mbf{T}_3 = \mbf{W}^{-1} \mbf{W}^{-1} \mbf{R}_0 \mbf{P} + \mbf{P} (\mbf{W}^{-1} \mbf{R}_0)^\trans \mbf{W}^{-1}$.
%
Finally, we are ready to state the convex optimization problems (namely, SDPs) that are involved in the iterative controller synthesis. The SDPs, along with the respective solutions, are given by
\small{
\begin{align} 
(\beta^\star, \mbf{R}^\star, \mbf{F}^\star) &= \argmin_{\beta, \mbf{R}, \mbf{F}} \bigl\{ \beta~|~ \eqref{eq:step1}, \eqref{eq:step1_Ksize}, \mbf{R} > 0, \beta > 0 \bigr\}, \label{eq:SDP-1} \\
\left( \beta^\star, \mbf{P}^\star, \{ \lambda_i^\star \} \right) &=
    \argmin_{\beta, \mbf{P}, \{ \lambda_i \}} \bigl\{ \beta ~|~ \eqref{eq:revTheorem1}, \eqref{eq:smallK_step2},  \mbf{P} > 0, \beta > 0, \nonumber \\
&  \qquad \qquad \qquad \quad  \lambda_i > 0, i \in \mathcal{N}_{n_w} \bigr\},
\label{eq:SDP-2} \\
\left( \beta^\star, \mbf{R}^\star, \mbf{F}^\star \right) &=
\argmin_{\beta, \mbf{R}, \mbf{F}} \bigl\{ \beta ~| ~ \eqref{eq:step3}, \eqref{eq:smallK_step3}, \mbf{R} > 0, \beta > 0 \bigr\} .  \label{eq:SDP-3}
\end{align} }
\normalsize
\setlength{\textfloatsep}{3pt}
\begin{algorithm}[t!]
{Initialization}: Choose $\mbf{W}_0, \mbf{r}, n_{\textnormal{max}}$. \\
 \For{$j=1:n_r$ \do} {
Set $r = r_j$ where $r_j$ is the $j$-th entry of $\mbf{r}$ and $\mbf{W} = \mbf{W}_0$. 
Set $\mathbb{X}_c = \mathcal{E}_{n_x}(\mbf{W})$ and $\mathbb{U}_c = r \mathcal{B}_{n_{u}}$.
 \Repeat {\textnormal{The largest $\mathbb{X}_c= \mathcal{E}_{n_x}(\mbf{W})$ is certified}} {
Get the sampled data points corresponding to $\mathbb{X}_c \times \mathbb{U}_c$. Use those to compute $\gamma_i, i \in \mathcal{N}_{n_w}$. \\
  Get $\beta^\star, \mbf{R}^\star, \mbf{F}^\star$ from \eqref{eq:SDP-1}. Set $\mbf{K} = \mbf{F}^\star (\mbf{R}^\star)^{-1}$, $\mbf{R}_0 = \mbf{R}^\star$, $c_t = 1$. \\
\While{$c_t \leq n_{\textnormal{max}}~\&~\bar{\sigma} (\mbf{K} \mbf{W}) \geq r$} {
       Get $\beta^\star, \mbf{P}^\star, \{ \lambda_i^\star \}$ from \eqref{eq:SDP-2} using $\mbf{R}_0$, $\mbf{K}$. Set $\mbf{R}_0 = (\mbf{P}^\star)^{-1}$ . \\
       Get $\beta^\star, \mbf{R}^\star, \mbf{F}^\star$ from \eqref{eq:SDP-3} using $\mbf{R}_0$, $\{ \lambda_i^\star \}$.  \\ 
    Set $\mbf{K} = \mbf{F}^\star (\mbf{R}^\star)^{-1}$, $\mbf{R}_0 = \mbf{R}^\star$, $c_t = c_t + 1$.
     }
     \eIf{$\bar{\sigma} (\mbf{K} \mbf{W}) \leq r$} {
         Update $\mbf{W}$ to get a larger $\mathbb{X}_c = \mathcal{E}_{n_x}(\mbf{W}) \subseteq \mathbb{X}$.
         }{
        Update $\mbf{W}$ to get a smaller $\mathbb{X}_c \hspace{-0.1cm} =  \mathcal{E}_{n_x}(\mbf{W}) \subseteq \mathbb{X}$.
        } 
        }
Output: $\mbf{K}$ and $\mbf{W}$.
 }
 \caption{Controller Synthesis Algorithm}
    \label{algorithm}
\end{algorithm}
Note that 
the SDP in \eqref{eq:SDP-2} requires known values of $\mbf{K}$ and $\mbf{R}_0$, similar to 
the SDP in \eqref{eq:SDP-3} which requires
the values of $\lambda_i$ and $\mbf{R}_0$ to be known.
Thus, 
the SDP in \eqref{eq:SDP-1} can be initially solved to obtain $\beta^\star, \mbf{R}^\star, \mbf{F}^\star$ and set $\mbf{K} = \mbf{F}^\star (\mbf{R}^\star)^{-1}$, $\mbf{R}_0 = \mbf{R}^\star$. The values of $\mbf{K}$ and $\mbf{R}_0$ can then be utilized in solving 
the SDP in \eqref{eq:SDP-2} to get $\beta^\star, \mbf{P}^\star, \{ \lambda_i^\star \}$.
Now, the tuple $(\mbf{P}, \mbf{K}, \tau, \{ \lambda_i \})$ with $\mbf{P} = \mbf{P}^\star$, $\tau = \sqrt{\beta^\star}$, $\{ \lambda_i \} = \{ \lambda_i^\star \}$  satisfies 
\eqref{eq:revTheorem1}, \eqref{eq:revTheorem2}. However, \eqref{eq:revTheorem3} might not hold and our approach involves iterating between 
\eqref{eq:SDP-2} and \eqref{eq:SDP-3} to satisfy $\bar{\sigma}(\mbf{KW}) \leq r$, if that is possible without modifying $\mbf{W}$.
The control synthesis starts from a small ellipsoid $\mathbb{X}_c = \mathcal{E}_{n_x}(\mbf{W})$ for a given $r>0$. Iterations are then carried out to certify the largest possible ellipsoid $\mathbb{X}_c = \mathcal{E}_{n_x}(\mbf{W})$ for that $r$, while simultaneously satisfying $\bar{\sigma}(\mbf{KW}) \leq r$. 
The overall procedure for control synthesis is summarized in Algorithm~\ref{algorithm}. Given the sampling region $\mathbb{X} \times \mathbb{U}$, Algorithm~\ref{algorithm} should be initialized by choosing a $\mbf{W} = \mbf{W}_0 > 0$ such that $\mathbb{X}_c = \mathcal{E}_{n_x}(\mbf{W}) \subseteq \mathbb{X}$ is sufficiently small. Also as a part of the initialization, $r_m = \max_{r>0} \{r~|~\mathbb{U}_c = r \mathcal{B}_{n_{u}} \subseteq \mathbb{U} \}$ should be determined to specify an $n_r$-point grid $\mbf{r} = (r_0, \ldots, r_m)$ where $r_0 > 0$ is chosen to be small. Finally, a maximum iteration number $n_{\text{max}}$ should be chosen when implementing Algorithm~\ref{algorithm}. 
Note that a variation of the Algorithm~\ref{algorithm} can be obtained where only \eqref{eq:SDP-1} is utilized (i.e., without the iteration between \eqref{eq:SDP-2} and \eqref{eq:SDP-3}). However, in our experience, this generally leads to more conservative results.
%
\section{Numerical Examples}
Two numerical examples are included in this section, concerning two-dimensional single-input systems. The nonlinearity explicitly depends on control inputs in the first example. The second example considers a system where the nonlinearity is a function of the states only. We choose
$\mbf{W}=\alpha \mbf{I}$, $\alpha>0$, meaning $\mathbb{X}_c  = \mathcal{E}_2 (\mbf{W})$ is a circle of radius $\alpha$ in these examples. Also, the SDPs in \eqref{eq:SDP-1}-\eqref{eq:SDP-3} are solved in MATLAB using \texttt{YALMIP} \cite{lofberg2004yalmip} and \texttt{MOSEK} \cite{mosek}. 
%
\begin{figure}[t!]
\centering \vspace{9pt}
\includegraphics[width=0.37\textwidth]{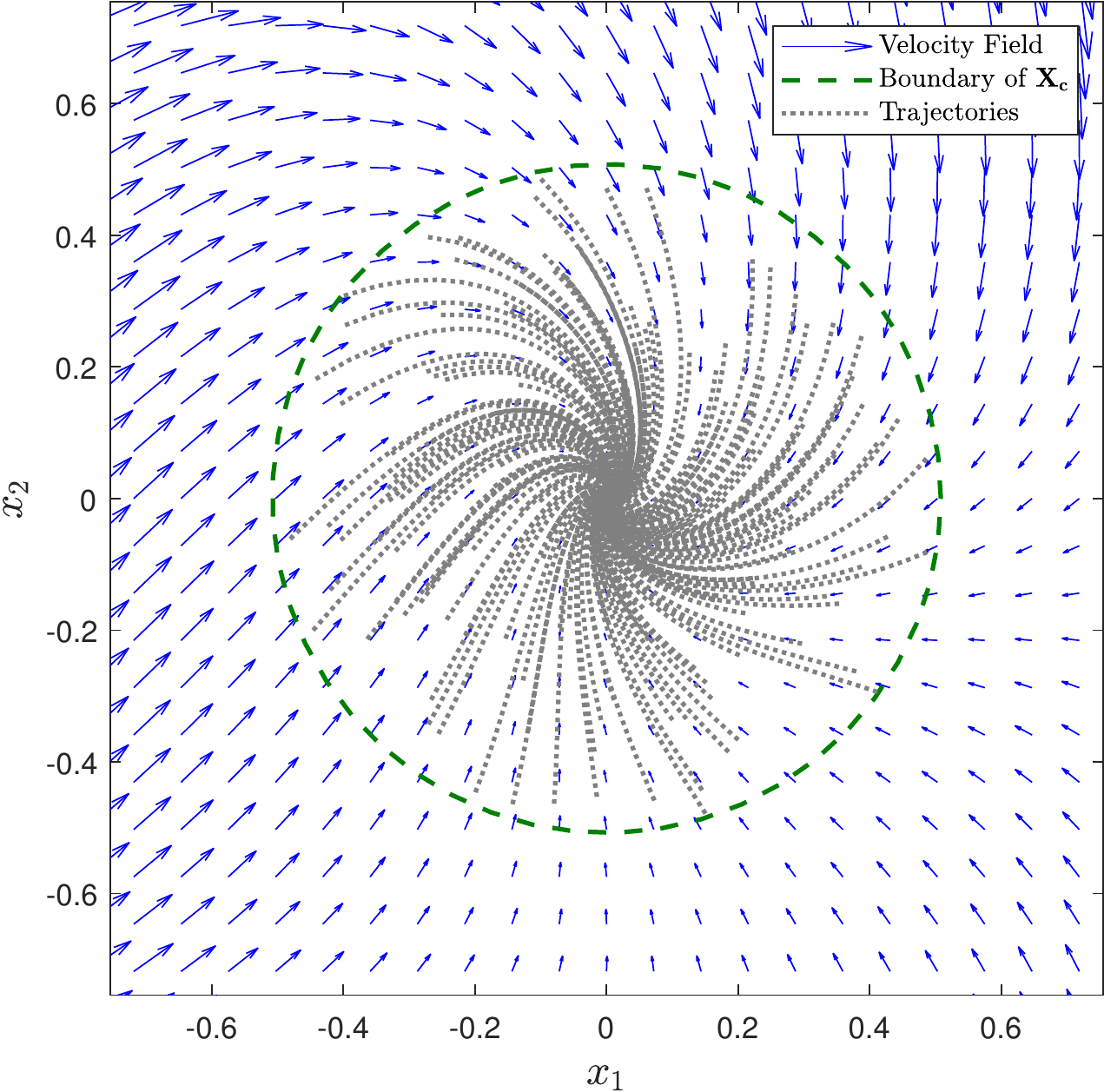}
\caption{Simulation results of the closed-loop system \eqref{eq:closed_loop_system_example_1} 
with $r=0.5$ and $\mbf{K}=\bbm -0.7151  & -0.6762 \ebm$.
} \label{fig:LMI_controller_1}
\end{figure}

\textit{Example-1:}
Consider a nonlinear system of the form 
\begin{equation} \label{eq:nonliner_system_example_1}
\dot{\mbf{x}} = \begin{bmatrix}
\dot{x}_1 \\ \dot{x}_2
\end{bmatrix} = \begin{bmatrix}
-0.1 x_1 + x_2 + u - x_1 x_2 + u^2 \\
-0.1 x_2 + u + x_1^2 - u^2
\end{bmatrix},
\end{equation}
with the corresponding 
equilibrium $(\mbf{x}_0, u_0) = (\mbf{0}, 0)$.
The nonlinear system in \eqref{eq:nonliner_system_example_1} is cast in the form of \eqref{eq:LTI system plus uncertainty} with  
\begin{equation*}
\begin{split}
    \mbf{A} &= \bbm -0.1 & 1 \\ 0 & -0.1\ebm , \ \mbf{B}_1 = \bbm 1 \\ 1\ebm, \\
    \mbs{\Delta}(\delta \mbf{x}, \delta u) &= \bbm -\delta x_1 \delta x_2 + \delta u^2 \\ \delta x_1^2 - \delta u^2 \ebm = \begin{bmatrix}
    \Delta_1 (\delta \mbf{x}, \delta {u}) \\ \Delta_2 (\delta \mbf{x}, \delta {u})
    \end{bmatrix}.
\end{split}
\end{equation*}
The system is then expressed in closed-loop form 
as
\begin{align} \label{eq:closed_loop_system_example_1}
    \delta \dot{\mbf{x}} &= (\mbf{A} + \mbf{B}_1 \mbf{K}) \delta \mbf{x} + \mbf{w}, \ \mbf{w} = [ \Delta_1 \ \Delta_2 ]^\trans, \\
    \mbf{v}_1 &= \left( \bbm \mbf{1}_{3_{1}} & \mbf{1}_{3_{2}} \ebm + \mbf{1}_{3_{3}} \mbf{K} \right) \delta \mbf{x},
    \mbf{v}_2 = \left( \bbm \mbf{1}_{3_{1}} & \mbf{0} \ebm + \mbf{1}_{3_{3}} \mbf{K} \right) \delta \mbf{x}. \nonumber
\end{align}
Now, Algorithm \ref{algorithm} is implemented with 11 values of $r$ between 0.01 and 0.5. %
The largest radius certified is $\alpha = 0.508$, which corresponds to $r=0.5$ and $\mbf{K}=\bbm -0.7151  & -0.6762 \ebm$. The simulation results of the closed-loop system with this controller are shown in the form of a phase portrait plot in Fig. \ref{fig:LMI_controller_1}. 
Closed-loop trajectories starting from different initial conditions in the set $\mathbb{X}_c$ 
converge to the origin, illustrating asymptotic convergence in the certified region. The velocity field indicates that 
the largest set $\mathbb{X}_c$ certified is contained within an even larger asymptotically stable region. This is likely due to the local norm bounds holding true for this larger region.
In summary, this example demonstrates that the proposed method is able to certify local asymptotic stability of the nonlinear system with state feedback, using only sampling and no explicit knowledge of the system's nonlinearities.
\begin{figure}[t!]
\centering \vspace{9pt}
\includegraphics[width=0.37\textwidth]{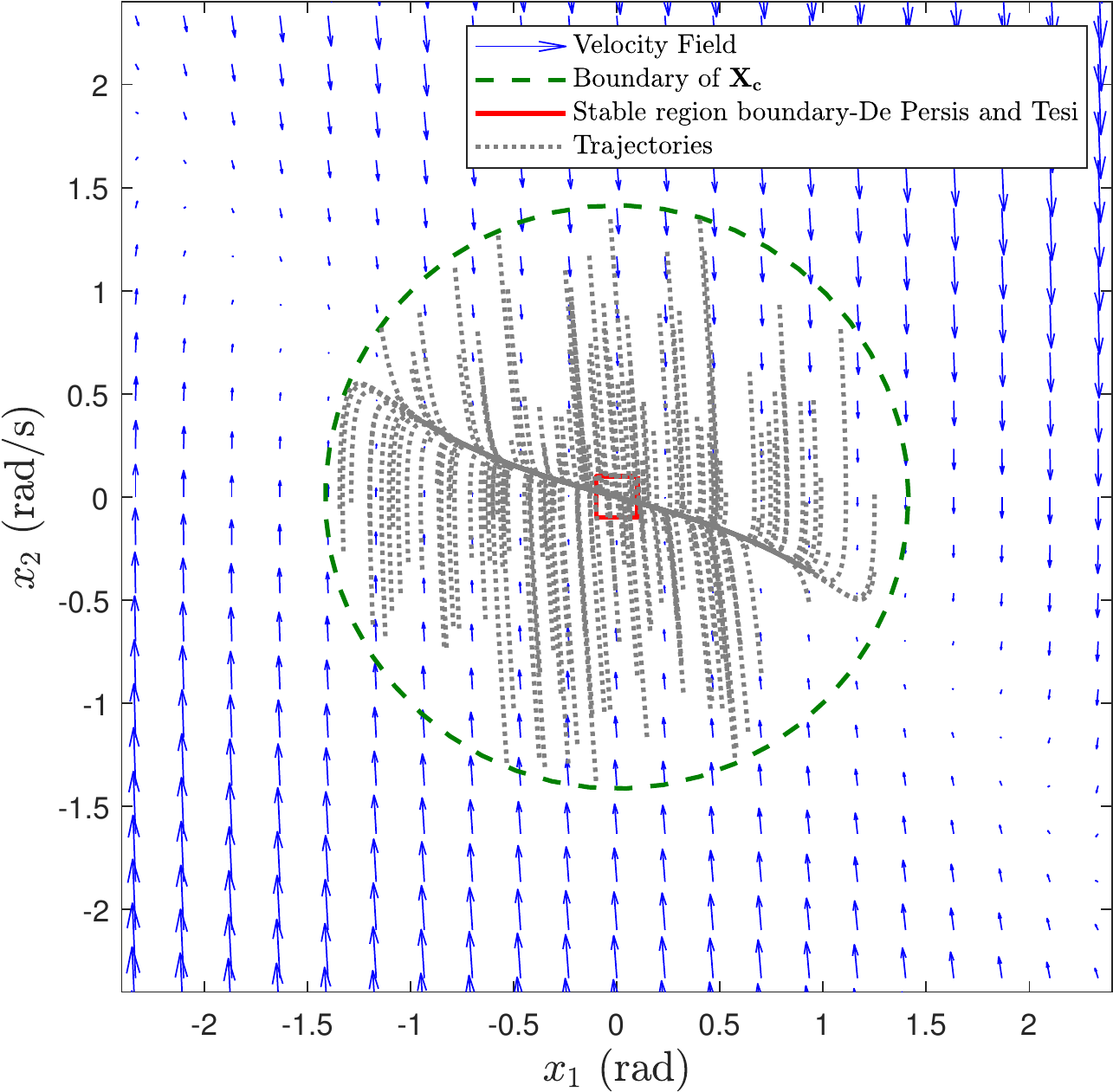}
\caption{Simulation results of the closed-loop system \eqref{eq:closed_loop_system_example_2} with 
$\mbf{K}=\bbm -13.4283 & -13.5242 \ebm$. } \label{fig:LMI_controller_2}
\end{figure}

\textit{Example-2:}
Consider the inverted pendulum example in \cite[Section V.B]{dePersis2020} in continuous time with the same unstable equilibrium at $(\mbf{x}_0,u_0)=(\mbf{0},0)$, which corresponds to the pendulum in upright position. The governing system is expressed in the
closed-loop form as
\begin{equation} \label{eq:closed_loop_system_example_2}
\resizebox{.89\hsize}{!}{$    \delta \dot{\mbf{x}} 
    = 
   \left( \bbm 0 & 1 \\ g & -\mu \ebm 
    +
    \mbf{1}_{2_{2}} \mbf{K} \right) \delta \mbf{x}
    +
    \bbm 0 \\
        \frac{g}{l} ( \sin{(\delta x_1)} - \delta x_1) 
    \ebm,$}
\end{equation}
\normalsize
where 
$l=1$, $g=9.8$, and $\mu=0.01$ (as in \cite{dePersis2020}). In this example, $\mbf{B}_2 = \mbf{1}_{2_{2}}$, $\mbf{C}_1=\bbm \mbf{1}_{3_{1}} & \mbf{0}\ebm$ and $\mbf{D}_1 = \mbf{0}$. Since the nonlinearity is independent of $\delta u$, we do not necessarily need to constrain the control input. However, performing the iterations in Algorithm \ref{algorithm} to reduce $\bar{\sigma}(\mbf{KW})$ in turn reduces the control effort required, and we let these iterations continue for $n_\text{max}=20$.  
In this setup, the maximum certified radius is $\alpha = \sqrt{2}$, along with the
controller $\mbf{K}=\bbm -13.4283 & -13.5242 \ebm$. This controller is therefore able to drive the pendulum to its upright position from an initial displacement of approximately 81 degrees. The simulation results of the system \eqref{eq:closed_loop_system_example_2} with this controller are depicted in Fig. \ref{fig:LMI_controller_2} where, similar to Example-1, the vector field indicates that the closed-loop system can be driven to the equilibrium from a much larger region than the certified region $\mathbb{X}_c$. Indeed, there appears to be a stable manifold with the vector field converging to it (see Fig.~\ref{fig:LMI_controller_2}).
Also in Fig. \ref{fig:LMI_controller_2}, the red square denotes the local region certified in \cite{dePersis2020}. In comparison, the proposed controller is able to certify a much larger region. 
This improvement was achieved, in part, by utilizing the knowledge of the nominal LTI system whereas, the controller in \cite{dePersis2020} is purely data-driven. This example thus demonstrates the efficacy of the proposed method over a purely data-driven framework, given the nominal LTI system is known.


\section{Conclusions and Future Work}
We presented an iterative method of local stabilization for nonlinear systems using sampled I/O data.
Our approach uses I/O data to derive local norm bounds and synthesizes a robust state-feedback controller that is guaranteed to stabilize the system within the sampling region for the set of nonlinearities satisfying the norm bounds.
The iterative steps require solving SDPs which can be done efficiently using freely available solvers. One of the numerical examples highlighted the reduced conservatism in our proposed synthesis method compared to a purely data-driven approach.
Our future efforts will involve introducing parametric uncertainties and exogenous signals into the proposed framework. Also, we will investigate other ways to characterize I/O behavior of the nonlinearities (e.g., weighted 2-norm bounds) and extend our formulation to the output-feedback case.

\bibliographystyle{IEEEtran}
\bibliography{Refs}

\end{document}